\documentclass[10pt,twoside,a4paper,reqno]{amsart}
\usepackage{amsfonts}
\usepackage{color}
\usepackage[colorlinks=true]{hyperref}
\usepackage{mathrsfs}
\makeatletter
\newtheorem{theorem}{Theorem}[section]
\newtheorem{definition}[theorem]{Definition}

\newtheorem{lemma}[theorem]{Lemma}

\numberwithin{equation}{section}
\begin{document}
\title[Nonnegative solutions for the fractional Laplacian]{Nonnegative solutions for the fractional Laplacian involving a nonlinearity with zeros}
\author[Salom\'on Alarc\'on]{Salom\'on Alarc\'on$^{\dag,1}$}
\author[Leonelo Iturriaga]{Leonelo Iturriaga$^{\dag,2}$}
\author[Antonella Ritorto]{Antonella Ritorto$^{\ddag,3}$}
\address[$\dag$]{Departamento de  Matem\'atica, Universidad T\'ecnica Federico Santa Mar\'ia, Casilla 110-V, Valpara\'iso, Chile.}
\address[$\ddag$]{Mathematical Institute, Utrecht University, Hans Freudenthalgebouw,
	Budapestlaan 6, 3584 CD, Utrecht, Netherlands.}
\email[$1$]{salomon.alarcon@usm.cl}
\email[$2$]{leonelo.iturriaga@usm.cl}
\email[$3$]{{\text{a.ritorto@uu.nl}}}
\thanks{All authors were partially supported by FONDECYT\,grant\,1171691, Chile. Besides, two first authors were partially supported by\,FONDECYT grants 1161635 and 1181125, Chile. }
\date{\today}
\maketitle
\begin{abstract}
We study the nonlocal nonlinear problem 
\begin{equation}\label{ppp}
\left\{
\begin{array}[c]{lll}
(-\Delta)^s u = \lambda f(u) & \mbox{in }\Omega, \\
u=0&\mbox{on } \mathbb{R}^N\setminus\Omega,
\end{array}
\right.
\tag{$P_{\lambda}$}
\end{equation}
where $\Omega$ is a  bounded smooth domain in $\mathbb{R}^N$\!,\,$N>2s$,\,$0<s<1$; $f:\mathbb{R}\rightarrow [0,\infty)$ is a nonlinear continuous function such that $f(0)=f(1)=0$ and $f(t)\sim |t|^{p-1}t$ as $t\rightarrow 0^+$, with $2<p+1<2^*_s$; and $\lambda$ is a positive parameter. 
We prove the existence of two nontrivial solutions $u_{\lambda}$ and $v_{\lambda}$ to (\ref{ppp}) such that $0\le u_{\lambda}< v_{\lambda}\le 1$ for  all sufficiently large $\lambda$.  
The first solution $u_{\lambda}$ is obtained by applying the Mountain Pass Theorem, whereas the second, $v_{\lambda}$,  via the sub- and super-solution method. 
We point out that our results hold regardless of the behavior of the nonlinearity $f$ at infinity. In addition, we obtain that these solutions belong to $L^{\infty}(\Omega)$.
\end{abstract}

\setcounter{equation}{0}
\section{Introduction}
This paper concerns with the existence of nonnegative solutions of the following nonlocal nonlinear elliptic problem
\begin{equation}\label{P}
\left\{
\begin{array}[c]{lll}
(-\Delta)^s u = \lambda f(u) & \mbox{in }\Omega, \\
u=0&\mbox{on } \mathbb{R}^N\setminus\Omega,
\end{array}
\right.
\end{equation}
where $\Omega$ is a  bounded smooth domain in $\mathbb{R}^N$, $N>2s$, $0<s<1$; $f:\mathbb{R}\rightarrow [0,\infty)$ is a nonlinear continuous function such that $f(0)=f(1)=0$ and $f(t)\sim |t|^{p-1}t$ as $t\rightarrow 0^+$, with $1<p<\frac{N+2s}{N-2s}=2_s^*-1$; and $\lambda$ is a positive parameter. 

In the case $s=1$, the local version of (\ref{P}) is reduced to the problem 
\begin{equation}\label{P2}
\left\{
\begin{array}[c]{lll}
-\Delta u = \lambda f(u)&\mbox{in } \Omega,\\
u = 0 &\mbox{on }\partial \Omega,
\end{array}\right.
\end{equation}
where $f$ is a nonnegative continuous function. 

When $f$ is positive, it is known that its behavior at zero and/or infinity can get to play a crucial role in the existence question of solutions. 
This situation can be clearly observed for the choice $f(t) = t^p$, $p>1$, in star-shape domains, on which we know that there exist positive solutions in $C^2(\Omega)\cap C(\overline{\Omega})$ to (\ref{P2}) if and only if $ p < \frac{N+2}{N-2}$. For fairly general nonlinearities, and using different approaches, many other authors have also studied the influence of the behavior at zero and/or infinity of $f$  in the existence question, see for example \cite{Ambrosetti-Brezis-Cerami,Ambrosetti-Hess,DeFigueiredo-Gossez-Ubilla,DeFigueiredo-Lions-Nussbaum} among other research papers.

When $f$ is nonnegative and has one zero, the existence issue of solutions to (\ref{P2}) is different, as was noted in \cite{Lions}. This study  has been extended to problems where the Laplacian was replaced by the $p$-Laplacian \cite{Iturriaga-Lorca-Massa, Iturriaga-Massa-Sanchez-Ubilla} or Pucci's operators \cite{Alarcon-Iturriaga-Quaas, Xiaohui}. 
In all these works was shown that there exist two positive solutions for sufficiently large $\lambda$ and by assuming some additional conditions on $f$. 
We also know two interesting contributions which have arisen recently. Firstly, by considering  $f$ nonnegative and having $r$ zeros, in \cite{GarciaMelian-Iturriaga}  was proved that (\ref{P2}) has $2r$ positive solutions provided only that $f$ verifies a suitable non-integrability condition near each of its zeros. Secondly, in \cite{Barrios-GarciaMelian-Iturriaga} was proved existence of positive solutions independently of the behavior of $f$ near zero or infinity, for $f$ Lipschitz having an isolated positive zero and verifying an additional growth hypothesis around such zero. 

In the case $0<s<1$, the existence of nonnegative solutions to (\ref{P}) also has been studied. Indeed, 
such as the local problem, the behavior at zero and/or infinity of $f$ also can get to exert an influence on the existence of nonnegative solutions. Again the choice $f(t) = t^p$, $p>1$, in star-shape domains, leads to nonexistence of bounded positive solutions, see \cite{RosOton-Serra}. We also know some recent research involving the fractional Laplacian, where some nonlinearities were considered. Based on a variational principle, in \cite{Kouhestani-Mahyar-Moameni} was proved multiplicity result when the nonlinearity leads to the well known convex-concave problem; whereas in  \cite{Biswas-Lorinczi} were considered nonlinearities exhibiting semi-linear and super-linear growth and  by using analytic and probabilistic tools, some Ambrosetti-Prodi type results were stablished. However, as far as our knowledge is concerned, there is no results in literature that addresses nonlinearities with zeros when $0 <s <1$.
Therefore, this is an interesting topic to investigate, which is the main objective of this paper.\medskip

To put into perspective our result, throughout this paper we consider  $f\colon \mathbb{R} \to [0, \infty)$ being a continuous function that verifies the following conditions: 
\begin{itemize}
	\item[($F_1$)] $\displaystyle\lim_{t\to 0^+}\frac{f(t)}{|t|^{p-1}t}=1$ for some $1<p<\frac{N+2s}{N-2s}$,
	\item[($F_2$)] $f(1)=0$ and $f(t)>0$ for $t\in(0,1)\cup(1,2)$,
	\item[($F_3$)] there exists $M_0>0$ such that the map $t\mapsto f(t)+M_0t$ is increasing for $t\in \mathbb{R}$. 
\end{itemize}
Observe that a such function $f$ verifies $f(0)=0$ due to $(F_1)$. 

Our main result is the following.

\begin{theorem}\label{main-theorem0}
Let $f\colon \mathbb{R} \to [0,\infty)$ be a continuous function verifying $(F_1)-(F_3)$. Then, there exists $\bar{\lambda}>0$ such that for $\lambda>\bar{\lambda}$ the problem $(P_{\lambda})$ admits two nonnegative solutions $u_{\lambda}, v_{\lambda} \in X^s_0(\Omega)$. Moreover, $0\le u_\lambda < v_{\lambda}\le 1$.
\end{theorem}

The main ideas behind of the proof relies on known arguments for solving this type of problems. However, this issue is non-trivial, by which we need to treat the problema under an appropriate approach that allows to apply such arguments and, in this way, to save several technical difficulties that arise in the nonlocal case.  

Specifically, we obtain the first solution $u _{\lambda}$ by truncating the nonlinearity and taking $\lambda$ large enough, where we have  assumed only the local condition at zero  ($F_1$).  The behavior of the norms, according to the parameter  $\lambda $ of the possible solutions, is based on the De Giorgi-Nash-Moser theory, precisely on Moser's iterative scheme. Such procedure was introduced in the mid 1950s and early 1960s, we refer to \cite{DeGiorgi,Moser}.   An important fact in the proof is that the functional energy associated with the truncated problem has the Mountain Pass (MP) geometry. Thus, we can control, in terms of the parameter $\lambda$, the MP level and the norm of the MP solution.

Respect to the second solution, we impose the conditions ($F_2$) and ($F_3$) on $f$, which widely known when applying the sub- and super-solution method in problems like (\ref{P}) with $s=1$. Here we show that this idea still remains hold by defining properly sub- and super-solutions in our nonlocal context, it is when $0<s<1$.

The document is organized as follows: in Section 2 we offer a brief review of the fractional spaces of Sobolev in the context of our problem and remember some useful results. In Section 3 we look for the first solution, while in Section 3 we find the second solution, which ends the proof of Theorem \ref{main-theorem0}

%

\setcounter{equation}{0}
\section{Functional framework and preliminaries}\label{section2}
%
In this section, we offer a brief review of the fractional Sobolev spaces in the context of our problem.  
Let $\mathscr{S}(\mathbb{R}^N)$ be the Schwartz space of rapidly decaying  smooth functions, {\it i.e.},
$$
\mathscr{S}(\mathbb{R}^N):=\Big\{ \vartheta\in C^{\infty}(\mathbb{R}^N)\,\vert\, \sup_{x\in\mathbb{R}^N}| x^{\alpha} D^{\beta}\vartheta(x)|<\infty\quad\mbox{for all }{\alpha},{\beta} \in \mathbb{N}^N_0\Big\}.
$$
Here, we are considering  the fractional Laplacian $(-\Delta)^s$, with $s\in(0,1)$, of a function $\vartheta \in \mathscr{S}(\mathbb{R}^N)$ defined in the {\it principal-value sense} as
\begin{align*}
(-\Delta)^s \vartheta(x) &:= c(N,s)\,\mbox{P.V.}\int_{\mathbb{R}^N} \frac{\vartheta(x)-\vartheta(y)}{|x-y|^{N+2s}} \, dy\\
&=-\frac{c(N,s)}{2}\int_{\mathbb{R}^n} \frac{\vartheta(x+z) - 2\vartheta(x) + \vartheta(x-z)}{|z|^{N+2s}}\, dz. 
\end{align*}
where $c(N,s):= \big( \int_{\mathbb{R^N}} \frac{1-\cos\zeta_1}{|\zeta|^{N+2s}}d\zeta \big)^{-1}$ is a normalization constant. 
We remark that the fractional Laplacian also can be viewed as a pseudo-differential operator of symbol $|\xi|^{2s}$ defined for any function $\vartheta$ in $\mathscr{S}(\mathbb{R}^N)$  as
$$
(-\Delta)^s \vartheta = \mathscr{F}^{-1}(|\xi|^{2s}\mathscr{F}(\vartheta)),
$$
where $\mathscr{F}$ denotes Fourier transform, {\it i.e.}, 
$$
\mathscr{F}(\vartheta)(\xi) = (2\pi)^{-\frac{N}{2}}\int_{\mathbb{R}^N} e^{-\mathrm{i}\xi \cdot x}\vartheta(x)\,dx \quad \mbox{for all }\xi\in\mathbb{R}^N,
$$
and $\mathscr{F}^{-1}$ its inverse, {\it i.e.}, 
$$
\mathscr{F}^{-1}(\widehat{\vartheta})(x) = (2\pi)^{-\frac{N}{2}}\int_{\mathbb{R}^N} e^{\mathrm{i} \xi\cdot x}\widehat{\vartheta}(\xi)\,d\xi \quad \mbox{for all }x\in\mathbb{R}^N,
$$
where $\widehat{\vartheta}:=\mathscr{F}(\vartheta)$, that verifies $\widehat{\vartheta}\in\mathscr{S}(\mathbb{R}^N)$.  See \cite[Proposition 3.3]{DiNezza-Palatucci-Valdinoci}.
	
We consider now the fractional Sobolev space $H^s(\Omega)$ defined as 
$$
H^s(\Omega)=\left\{h\in L^2(\Omega)\,\vert\, \frac{|h(x)-h(y)|}{|x-y|^\frac{N+2s}{2}}\in L^2(\Omega\times\Omega)\right\}
$$
endowed with the norm
$$
\|h\|_{H^s(\Omega)}:=\left(\int_{\Omega}|h(x)|^2\,dx+\int_{\Omega\times \Omega}\frac{|h(x)-h(y)|^2}{|x-y|^{N+2s}}\,dx\,dy\right)^\frac{1}{2},
$$
where the term 
$$
[h]_{H^s(\Omega)}:=\left(\int_{\Omega\times\Omega}\frac{|h(x)-h(y)|^2}{|x-y|^{N+2s}}\,dx\,dy\right)^\frac{1}{2},
$$
is the so-called Gagliardo seminorm of $h$. 
Also we consider the space $H_0^s(\Omega)$ which denotes the closure of $C_0^\infty(\Omega)$ with respect to the norm $\|\cdot\|_{H^s(\Omega)}$.

Since we are considering homogeneous Dirichlet boundary conditions, we need to work in a suitable functional analytical setting in order to correctly encode the Dirichlet datum in the variational formulation. 
In this way, it is convenient to introduce  the set 
$$
Q:=(\mathbb{R}^N\times\mathbb{R}^N)\setminus(\Omega^c\times\Omega^c)
$$
where $$\Omega^c:=\mathbb{R}^N\setminus\Omega,$$and consider the function $K:\mathbb{R}^N\setminus\{0\}\to(0,\infty)$ that is defined as 
$$
K(x):=|x|^{-(N+2s)}.
$$
We also consider the set $X^s(\Omega)$ being the linear space of all Lebesgue measurable functions from $\mathbb{R}^N$ to $\mathbb{R}$ such that the restriction to $\Omega$ of any function $u$ belongs to $L^2(\Omega)$, and the map $(x,y)\mapsto(u(x)-u(y))\sqrt{K(x-y)}$ belongs to  $L^2(Q)$. 
Here we consider to the space $X^s(\Omega)$ endowed with the norm 
$$
\|u\|_{X^s(\Omega)}=\|u\|_{L^2(\Omega)}+\left(\int_Q |u(x)-u(y)|^2K(x-y)\,dx\,dy\right)^\frac{1}{2}.
$$ 
We remark that $\|\cdot\|_{H_0^s(\Omega)}$ and $\|\cdot\|_{X^s(\Omega)}$ are not the same space because $\Omega\times\Omega$ is strictly contained in $Q$, so that this makes the classical fractional Sobolev space approach not sufficient for studying the nonlocal problem.
	
Now, we introduce the space
$$
X_0^s(\Omega):=\{u\in X^s(\Omega):\,u=0\mbox{ a.e. in }\Omega^c\}
$$
endowed with the norm induced by the norm of $X^s(\Omega)$; that is, 
$$
\|u\|_{X^s_0(\Omega)}:=\|u\|_{L^2(\mathbb{R}^N)}+\left(\int_{\mathbb{R}^N\times\mathbb{R}^N} |u(x)-u(y)|^2K(x-y)\,dx\,dy\right)^\frac{1}{2}.
$$
Note that  $\|\cdot \|_{X^s_0(\Omega)}$ is truly a norm in $X_0^s(\Omega)$, since for every $u\in X_0^s(\Omega)$, one has that  $u=0$ a.e. in $\Omega^c$.
	
Before continuous we recall some properties of the fractional Sobolev space that we will use in next sections (see for example \cite{Molica-Radulescu-Servadei}).

\begin{lemma}\label{lemma-inclusion}
Let $s\in(0,1)$, $N>2s$, and $\Omega$ an open bounded domain in $\mathbb{R}^N$ with $C^{0,1}$-boundary. 
Then, the following assertions hold:
\begin{itemize}
\item[(i)] There exists a positive constant $C$, depending only on $N$ and $s$, such that for any $u\in X_0^s(\Omega)$,
$$
\|u\|_{L^{2^*_s}(\Omega)}^2=\|u\|_{L^{2^*_s}(\mathbb{R}^N)}^2\leq C\int_{\mathbb{R}^N\times\mathbb{R}^N} |u(x)-u(y)|^2K(x-y)\,dx\,dy.
$$ 	
\item[(ii)] There exists a constant $C>1$, depending only on $N$, $s$ and $\Omega$, such that for any $u\in X_0^s(\Omega)$,
\begin{align*}
\int_Q |u(x)-u(y)|^2K(x-y)\,dx\,dy&\leq \|u\|_{X^s(\Omega)}^2\\ &\leq C\int_Q |u(x)-u(y)|^2K(x-y)\,dx\,dy,
\end{align*}
that is, 
$$
\|u\|_{X^s_0}(\Omega)=\left(\int_{\mathbb{R}^N\times \mathbb{R}^N} |u(x)-u(y)|^2K(x-y)\,dx\,dy\right)^{\frac12}
$$
is a norm in $X_0^s(\Omega)$ equivalent to the usual norm. 
\item[(iii)]  The embedding $X_0^s(\Omega)\hookrightarrow L^r(\Omega)$ is continuous for any $r\in[1, 2^*_s]$, and compact for any $r\in[1, 2^*_s)$.
\end{itemize}	
\end{lemma}

For convenience, from now on we consider the space $X_0^s(\Omega)$ endowed with the equivalent norm
$$
\|u\|_{X_0^s(\Omega)} := \left( \frac{1}{2} \int_{\mathbb{R}^N\times \mathbb{R}^N}|u(x)-u(y)|^2 K(x-y)\, dx\,dy    \right)^{\frac{1}{2}}.
$$

Finally, we give the notion of weak solution that we use throughout this paper. We say that $u\in X^s_0(\Omega)$ is a weak solution to $(\ref{P})$ if 
$$
\frac{1}{2} \int_{\mathbb{R}^N\times \mathbb{R}^N}(u(x)-u(y))(v(x)-v(y))K(x-y)\, dx\,dy = \lambda \int_{\Omega} f(u) v \, dx 
$$
for every $v\in X^s_0(\Omega)$.
%

\setcounter{equation}{0}
\section{The first solution}\label{section3} 

Let $f: \mathbb{R} \rightarrow [0,\infty)$ be a continuous function. 
Throughout this section, we only assume hypothesis $(F_1)$ on $f$. 
The principal outcome here is the next theorem.

\begin{theorem}\label{main}
Let $f: \mathbb{R} \rightarrow [0,\infty)$ be a continuous function such that $(F_1)$ is verified. 
Then, there exists $\lambda^*>0$ such that for any $\lambda>\lambda^*$ the problem (\ref{P}) admits a nonnegative solution $u_\lambda$. 
Moreover, 
$$
\lim_{\lambda \to +\infty}\|u_\lambda\|_{L^\infty(\Omega)}= 0.
$$
\end{theorem}
To prove Theorem \ref{main}, we first consider an auxiliary problem. 
Since we are looking for solutions $u_{\lambda}$ close to zero in the $L^{\infty}$-norm, we can truncate the problem (\ref{P}) as follows. 

For $R\in(0,1)$,  we consider the truncate problem
\begin{equation}\label{PLR}
\left\{\begin{array}[c]{lll}
(-\Delta)^s u = \lambda f_R(u) &\mbox{in }\Omega, \\
u=0&\mbox{on } \Omega^c,
\end{array}\right.
\end{equation}
where
$$
f_R(t)=
\left\{
\begin{array}[c]{lll}
f(t^+) &\mbox{ if }|t|\leq R,\smallskip\\ 
\displaystyle \frac{f(R)}{R^p}(t^+)^p &\mbox{ otherwise},
\end{array}
\right.
$$
and $t^+=\max\{0, t \}$. 
Note now that problem (\ref{PLR}) has a variational structure. 
Indeed, its weak formulation  is given by 
$$
\!\!\!\!\!\!\!\!\!\!\left\{\!\!\!
\begin{array}[c]{lll}
\displaystyle
\frac{1}{2}\!\int_{\mathbb{R}^N\times\mathbb{R}^N}\! \!(u(x)\!-\!u(y))(v(x)\!-\!v(y))K(x\!-\!y)dx\,dy\!=\!\lambda\!\int_{\Omega}\!f_R(u)v(x)
\,dx\quad\!\forall v\!\in\! X_0^s(\Omega),\smallskip\\ 
u\in X_0^s(\Omega),
\end{array}\!\!\!\!\!
\right.\!\!\!\!\!
$$
and its associated energy functional $J_{R,\lambda}:X_0^s(\Omega)\rightarrow \mathbb{R}$ is given by
\begin{equation}\label{J-lambda-R}
J_{R,\lambda}(u)=\frac{1}{4}\int_{\mathbb{R}^N\times\mathbb{R}^N} |u(x)-u(y)|^2 K(x-y)\,dx\,dy-\lambda\int_{\Omega}F_R(u(x))\,dx,
\end{equation}
where $F_R$ is the primitive of $f_R$, that is, $F_R(t)=\int_0^t f_R(\tau)\,d\tau$. 

Next lemma shows that if (\ref{PLR}) has a solution, then it is nonnegative.

\begin{lemma} Let $R \in (0, 1)$.
Assume that problem (\ref{PLR}) admits a solution $u_{\lambda,R}\in X_0^s(\Omega)$, then $u_{\lambda,R}\ge 0$ a.e. in $\Omega$.
\end{lemma}
\begin{proof}
If $u_{\lambda,R}\in X_0^s(\Omega)$, then $u_{\lambda,R}\in X^s(\Omega)$ and $u_{\lambda,R}=0$ a.e. $\Omega^c$. To simplify notation, let us denote $v= u_{\lambda,R}$. Then $v^+, v^- \in X^s_0(\Omega)$. Taking $v^-$ as a test function in the weak formulation of (\ref{PLR}) and, since
$$
\begin{array}[c]{lll}
(v(x)-v(y))(v^-(x)-v^-(y))\smallskip\\
\qquad =(v^+(x)-v^-(x)-v^+(y)+v^-(y))(v^-(x)-v^-(y))\smallskip\\ 
\qquad =-(v^-(x))^2-v^+(y)v^-(x)+v^-(y)v^-(x)\smallskip\\
\qquad \qquad -v^+(x)v^-(y) +v^-(x)v^-(y)-(v^-(y))^2\smallskip\\
\qquad =-(v^-(x)-v^-(y))^2-(v^+(y)v^-(x)+v^+(x)v^-(y)),
\end{array}
$$
then we get
\begin{align*}
0&=\int_{\Omega}f_R(-v^-(x))v^-(x)\,dx\\ &=\int_{\Omega}f_R(v(x))v^-(x)\,dx\\  &=\int_{\mathbb{R}^N\times\mathbb{R}^N} (v(x)-v(y))(v^-(x)-v^-(y))K(x-y)\,dx\,dy\\ &=-\int_{\mathbb{R}^N\times\mathbb{R}^N} (v^-(x)-v^-(y))^2K(x-y)\,dx\,dy\\ &\qquad-\int_{\mathbb{R}^N\times\mathbb{R}^N} (v^+(y)v^-(x)+v^+(x)v^-(y))K(x-y)\,dx\,dy,
\end{align*} 
which implies that $\|v^-\|_{X_0^s(\Omega)}=0$, because $v^+(y)v^-(x)+v^+(x)v^-(y)\geq 0$ for a.e. $(x,y)\in\mathbb{R}^N\times\mathbb{R}^N$, that leads to $v\geq 0$ a.e. in $\Omega$.
\end{proof}

In the sequel, we verify that the energy functional $J_{R,\lambda}$ given by \eqref{J-lambda-R} has the Mountain Pass geometry for all $R$ sufficiently small. 

\begin{lemma}\label{lema-geometria1}
Let $\lambda>0$, and $J_{\lambda,R}$ be given by \eqref{J-lambda-R}.  Then, for each sufficiently small $R$, there exist positive numbers $\rho_\lambda$ and $\beta_\lambda$ such that
\begin{itemize}
\item[(i)] $J_{R,\lambda}(u)\geq\beta_\lambda$ for any $\|u\|_{X_0^s(\Omega)}=\rho_\lambda$. Moreover, 
$$
\lim_{\lambda \to +\infty}\rho_\lambda = 0 = \lim_{\lambda \to +\infty} \beta_{\lambda}.
$$ 
\item[(ii)] There exists a function $e\in X_0^s(\Omega)$ such that $\|e\|_{X_0^s(\Omega)}>\rho_\lambda$ and $J_{R,\lambda}(e)<0$.
\end{itemize}   
\end{lemma}
\begin{proof}\hspace{1cm} 
\begin{itemize} 
\item[(i)] From $(F_1)$,  there exists a constant $\alpha_1>0$ such that $F_R(t)\leq \alpha_1|t|^{p+1}$,  for all  sufficiently small $R$. Then, by (iii) of Lemma \ref{lemma-inclusion}, we get
\begin{align*}
J_{R,\lambda}(u)&=\frac{1}{4}\int_{\mathbb{R}^N\times\mathbb{R}^N} |u(x)-u(y)|^2 K(x-y)\,dx\,dy-\lambda\int_{\Omega}F_R(u(x))\,dx\\ &\geq \frac{1}{4}\int_{\mathbb{R}^N\times\mathbb{R}^N} |u(x)-u(y)|^2 K(x-y)\,dx\,dy-\alpha_1\lambda\int_{\Omega}|u(x)|^{p+1}\,dx\\ &\geq \frac{1}{4}\|u\|_{X_0^s(\Omega)}^2-\lambda C\|u\|_{X_0^s(\Omega)}^{p+1}\\ &=\|u\|_{X_0^s(\Omega)}^2\left(\frac{1}{4}-\lambda C\|u\|_{X_0^s(\Omega)}^{p-1}\right).
\end{align*}
Hence, by taking $\rho_\lambda=(8\lambda C)^{\frac{-1}{p-1}}$ and $\beta_\lambda=\frac{\rho^2_\lambda}{8}$, we conclude the proof of assertion (i).   

\item[(ii)] Take $u\in X_0^s(\Omega)\setminus\{0\}$, $u\geq 0$, and consider $\varphi(t)=J_{R,\lambda}(tu)$, for $t\in \mathbb{R}$. It follows that  
$$\lim_{t\to +\infty}\varphi(t)= -\infty.$$ Then by choosing  $t_0>0$ sufficiently large so that $\varphi(t_0)<0$ and $e=t_0u\notin  B_{\rho_\lambda}(0)$, assertion (ii) holds.  
\end{itemize} 
\end{proof}	

As already mentioned before, the energy functional $J_{R,\lambda}$ of the truncated problem $(\ref{PLR})$ has the geometry of the Mountain Pass Theorem for all sufficiently small $R$. 
Since $f_R$ is a purely power for large values, it is not difficult to prove that $J_{R,\lambda}$ is a $C^1$-functional which satisfy the Palais-Smale condition. 
Consequently, we get the next existence result.

\begin{lemma}\label{lemma-conseceuncia-montana}
Let $\lambda>0$. Then, for each sufficiently small $R$ the problem (\ref{PLR}) admits a nontrivial nonnegative solution $u_{\lambda,R}\in X_0^s(\Omega)$.
\end{lemma}
Note that the function $u_{\lambda,R} \in X_0^s(\Omega)$ in Lemma \ref{lemma-conseceuncia-montana} satisfies 
\begin{equation}\label{u-lambda-R}
J_{R,\lambda}(u_{\lambda,R})=c_{\lambda, R} \quad  \text{ and }\quad J_{R,\lambda}'(u_{\lambda,R})=0,
\end{equation}
where $c_{\lambda,R}$ corresponds to the Mountain Pass level, that is,
$$
c_{\lambda,R}:=\inf\Big\{\sup_{t\in [0,1]} J_{R,\lambda}(\gamma(t)) \colon \gamma \in C([0,1]; X_0^s(\Omega)), \gamma(0)=0, \gamma(1)=u_{\lambda,R}\Big\}.
$$
%
\begin{lemma}\label{mplevel}
Let $\lambda>0$. If $R>0$ is sufficiently small, then
\begin{equation}\label{ML}
\lim_{\lambda \to +\infty}c_{\lambda,R}=0.
\end{equation}
\end{lemma}
\begin{proof} 
From $(F_1)$, we have that there exists a constant $d>0$ such that $$F_R(t)\geq d\, t^{p+1}\quad \mbox{for all } t\ge 0\mbox{ and for all sufficiently small $R$}.$$ Consider now  the positive constant $\beta_{\lambda}$ and  the function $e\in X_0^s(\Omega)$ given in Lemma \ref{lema-geometria1}. Since $0<\beta_\lambda\leq c_{\lambda,R}$, from Lemma \ref{lema-geometria1} (ii), we get
\begin{align*}
c_{\lambda,R}&\leq \max_{t\in[0,\infty]}\left(\frac{1}{4}\|te\|_{X_0^s(\Omega)}^2-\lambda\int_\Omega F_R(te)\right) \\ &\leq \max_{t\in[0,\infty]}\left(\frac{1}{4}\|e\|^2_{X_0^s(\Omega)}t^2-\lambda d\|e\|^{p+1}_{L^{p+1}(\Omega)}t^{p+1}\right)\\ &\leq t_{\lambda}^2 \frac{\|e\|^2_{X_0^s(\Omega)}}{4}\frac{(p-1)}{p+1},
\end{align*}
where $t_{\lambda}^{p-1}=\frac{\|e\|^2_{X_0^s(\Omega)}}{2\lambda\|e\|_{L^{p+1}(\Omega)}^{p+1} (p+1)d}$. Since $t_{\lambda}\rightarrow 0$ as $\lambda\rightarrow +\infty$, we immediately can deduce that (\ref{ML}) holds, which completes the proof.	
\end{proof}

We now are interested in establishing  the convergence rate of the solution $u_{\lambda,R}$ of (\ref{PLR}) in terms of the Mountain Pass level $c_{\lambda,R}$. 
The following lemma points in that direction.
\begin{lemma}\label{mpsol}
Let $\lambda >0$ and $u_{\lambda,R}\in X_0^s(\Omega)$  the function obtained in \eqref{u-lambda-R}. 
Then, there exists $0<\bar{R}<1$ such that $\|u_{\lambda,R}\|_{X_0^s(\Omega)}={O}(c_{\lambda,R}^\frac{1}{2})$ for any $0<R\leq\bar{R}$. In particular, if $0<R\leq\bar{R}$, then 
$$
\lim_{\lambda\to +\infty}\|u_{\lambda,R}\|_{X_0^s(\Omega)}=0.
$$
\end{lemma}
\begin{proof}
Thanks to $(F_1)$, we have that there exists $0<\bar{R}<1$ such that for any $0<R\leq\bar{R}$, there are positive constants  $\alpha_0$ and $\alpha_1$ such that $(p+1){\alpha}_0>4\alpha_1$, and 
$$
\alpha_0 t^p\leq f_R(t)\leq \alpha_1 t^p\quad \mbox{for all $t\geq 0$}.
$$
Notice that since $u_{\lambda,R}$ is a weak solution of (\ref{PLR}), then 
$$
\frac{1}{2}\int_{\mathbb{R}^N\times\mathbb{R}^N} (u_{\lambda,R}(x)-u_{\lambda,R}(y))(v(x)-v(y))K(x-y)\,dx\,dy=\lambda\int_{\Omega}f_R(u_{\lambda,R})v(x)\,dx
$$
for every $v\in X_0^s(\Omega)$, and
\begin{align*}
J_{R,\lambda}(u_{\lambda,R})&=c_{\lambda,R}\\
&=\frac{1}{4}\int_{\mathbb{R}^N\times\mathbb{R}^N} |u_{\lambda,R}(x)-u_{\lambda,R}(y)|^2 K(x-y)\,dx\,dy-\lambda\int_{\Omega}F_R(u_{\lambda,R}(x))\,dx.
\end{align*}
Then by taking $v=u_{\lambda,R}$ and by combining the two previous equalities, for any $\frac{\alpha_1}{\alpha_0(p+1)}\leq\theta<\frac{1}{4}$ we get 
\begin{align*}
c_{\lambda,R}&=\left(\frac{1}{4}-\theta\right)\int_{\mathbb{R}^N\times\mathbb{R}^N} |u_{\lambda,R}(x)-u_{\lambda,R}(y)|^2 K(x-y)\,dx\,dy\\ &\quad-\lambda\int_{\Omega}F_R(u_{\lambda,R}(x))\,dx+\lambda\theta\int_{\Omega}f_R(u_{\lambda,R})u_{\lambda,R}&\\ &\geq \left(\frac{1}{4}-\theta\right)\|u_{\lambda,R}\|_{X_0^s(\Omega)}^2+\lambda\left(\theta \alpha_0-
\frac{\alpha_1}{p+1}\right)\int_{\Omega}|u_{\lambda,R}|^{p+1}\,dx\\ &\geq \left(\frac{1}{4}-\theta\right)\|u_{\lambda,R}\|_{X_0^s(\Omega)}^2,
\end{align*}
which finishes the proof.
\end{proof}

Next lemma follows the arguments given in \cite{Brasco-Lindgren-Parini} (see also \cite{Iturriaga-Lorca-Montenegro}), and it gives us the boundedness in the $L^{\infty}$-norm of any solution in $X_0^s(\Omega)$ of the problem (\ref{PLR}). 
\begin{lemma}\label{mos}
If $u\in X_0^s(\Omega)$ is a nonnegative solution of (\ref{PLR}), then the $u$ belongs to $L^{\infty}(\Omega)$. 
Moreover, there exists a constant $L>0$ such that 
$$
\|u\|_{L^{\infty}(\Omega)}\leq \lambda^\frac{1}{2_s^*-p-1} L\|u\|_{L^{2^*_s}(\Omega)}^\frac{2_s^*-2}{2_s^*-p-1}.
$$
\end{lemma}
\begin{proof}
Since $(F_1)$ is verified, it is easy to check that there exists $\rho>0$ such that $|f_R(t)|\leq \rho|t|^p$ for all $t\in\mathbb{R}$. 

Let $M>0$ given and consider $u_M:=\min\{u,M\}$. 
Note that $u_M$ belongs to $X_0^s(\Omega)$ because it is just the composition of $u$ with a Lipschitz function.  
Consider now the function $v=u_M^{2k+1}$ for $k\in \mathbb{N}_0$ given, as a test function in the equation that corresponds to the weak formulation to problem (\ref{PLR}) verified by $u$. 
Then,
\begin{align*}
\frac{1}{2}\int_{\mathbb{R}^N\times\mathbb{R}^N} &(u(x)-u(y))(u_M^{2k+1}(x)-u_M^{2k+1}(y))K(x-y)\,dx\,dy\\ &=\lambda\int_{\Omega}f_R(u)u_M^{2k+1}(x)\,dx\\ 
&\leq \lambda\rho\int_{\Omega}u^{p+{2k+1}}(x)\,dx,
\end{align*}
where we have used that $u_M\leq u$. By using inequality (C.2) of \cite{Brasco-Lindgren-Parini}, we obtain 
\begin{align*}
\frac{1}{2}\int_{\mathbb{R}^N\times\mathbb{R}^N} &(u(x)-u(y))(u_M^{2k+1}(x)-u_M^{2k+1}(y))K(x-y)\,dx\,dy\\ &\geq \frac{4({2k+1})}{({2k+1}+1)^2}\int_{\mathbb{R}^N\times\mathbb{R}^N} |u_M^{k+1}(x)-u_M^{k+1}(y)|^2K(x-y)\,dx\,dy\\ &\geq \frac{(2k+1)}{({k+1})^2}C\left(\int_{\mathbb{R}^{N}}(u_M^{k+1})^\frac{2N}{N-2s}\right)^\frac{N-2s}{N},
\end{align*}
and  by H\"older's inequality, we get
\begin{align*}
\left(\int_{\mathbb{R}^{N}}(u_M^{k+1})^\frac{2N}{N-2s}\right)^\frac{N-2s}{N}&\leq \lambda\rho\frac{({k+1})^2}{C(2k+1)}\int_{\Omega}u^{p+{2k+1}}(x)
\,dx\\ &=\lambda\rho\frac{({k+1})^2}{C(2k+1)}\int_{\Omega}u^{p-1}u^{2(k+1)}\,dx\\ &\leq \lambda\rho\frac{({k+1})^2}{C(2k+1)}\left(\int_{\Omega}u^{2(k+1)l}dx\right)^\frac{1}{l}\left(\int_{\Omega}u^{2^*_s}dx\right)^\frac{p-1}{2_s^*},
\end{align*}
where $l=\frac{2_s^*}{2_s^*-(p-1)}$. Taking the limit as $M$ goes to $+\infty$, we obtain
\begin{align*}
\|u\|_{L^{2_s^*(k+1)}(\mathbb{R}^N)}&\leq \left(\frac{\rho\lambda(k+1)^2}{C(2k+1)}\right)^\frac{1}{2(k+1)}\|u\|_{L^{2l(k+1)}(\Omega)}\|u\|_{L^{2_s^*}(\Omega)}^\frac{p-1}{2(k+1)}.
\end{align*}
Defining $k_1\in \mathbb{R}$ such a way that $2 l(k_1+1)=2_s^*$, that is, 
$$
k_1+1=\frac{2_s^*}{2l}=\frac{2_s^*-(p-1)}{2}>1,
$$
it follows that
\begin{align*}
\|u\|_{L^{2_s^*(k_1+1)}(\mathbb{R}^N)}&\leq \left(\frac{\rho\lambda(k_1+1)^2}{C(2k_1+1)}\right)^\frac{1}{2(k_1+1)}\|u\|_{L^{2_s^*}(\Omega)}\|u\|_{L^{2_s^*}(\Omega)}^\frac{p-1}{2(k_1+1)}\\ &=\left(\frac{\rho\lambda(k_1+1)^2}{C(2k_1+1)}\|u\|_{L^{2_s^*}(\Omega)}^{p-1}\right)^\frac{1}{2(k_1+1)}\|u\|_{L^{2_s^*}(\Omega)}.
\end{align*}
Now, we proceed by induction  as $2l(k_n+1)=2_s^*(k_{n-1}+1)$, then $k_n+1=\left(\frac{2_s^*}{2l}\right)^n$ and
\begin{equation}\label{moser}
\begin{split}
\|u\|_{L^{2_s^*(k_n+1)}(\mathbb{R}^N)}&\leq \left(\frac{\rho\lambda(k_n+1)^2}{C(2k_n+1)}\right)^\frac{1}{2(k_n+1)}\|u\|_{L^{2l(k_n+1)}(\Omega)}\|u\|_{L^{2_s^*}(\Omega)}^\frac{p-1}{2(k_n+1)}\\
&= \left(\frac{\rho\lambda(k_n+1)^2}{C(2k_n+1)} \|u\|_{L^{2_s^*}(\Omega)}^{p-1}\right)^\frac{1}{2(k_n+1)}\|u\|_{L^{2l(k_n+1)}(\Omega)}\\ 
&= \left(\frac{\rho\lambda(k_n+1)^2}{C(2k_n+1)} \|u\|_{L^{2_s^*}(\Omega)}^{p-1}\right)^\frac{1}{2(k_n+1)}\|u\|_{L^{2^*_s(k_{n-1}+1)}(\Omega)}\\
&\leq  \prod_{i=1}^n\left(\frac{\rho\lambda(k_i+1)^2}{C(2k_i+1)} \|u\|_{L^{2_s^*}(\Omega)}^{p-1}\right)^\frac{1}{2(k_i+1)} \|u\|_{L^{2_s^*}(\Omega)}\\
&\leq  \prod_{i=1}^n\left(\frac{\rho\lambda(k_i+1)^2}{C(2k_i+1)} \right)^\frac{1}{2(k_i+1)}\|u\|_{L^{2_s^*}(\Omega)}^{1+\frac{p-1}{2}\sum_{i=1}^n\frac{1}{k_i+1}}.
\end{split}
\end{equation}
Setting 
$$
C_1=\rho^\frac{1}{2_s^*-p-1}C^\frac{-1}{2_s^*-p-1}\lim_{n\to\infty}\prod_{i=1}^n\left(\frac{(k_i+1)^2}{2k_i+1}\right)^\frac{1}{2(k_i+1)},
$$
and letting $n\to\infty$ in \eqref{moser} we obtain
$$
\|u\|_{L^\infty(\mathbb{R}^N)}\leq\lambda^\frac{1}{2_s^*-p-1}C_1\|u\|_{L^{2_s^*(\Omega)}}^\frac{2_s^*-2}{2_s^*-p-1}.
$$
\end{proof}

Now, we have all ingredients to conclude the existence of {the first} solution to the problem $(\ref{P})$ if $\lambda$ is chosen to be sufficiently large.
\begin{proof}[Proof of Theorem \ref{main}]
By combining Lemmas \ref{lemma-conseceuncia-montana}, \ref{mplevel}, \ref{mpsol} and \ref{mos} we get that for each sufficiently small $R$,
$$
\|u_{\lambda,R}\|_{L^\infty(\Omega)}\leq C\lambda^\frac{-1}{p-1},
$$
where $u_{\lambda,R}$ verifies (\ref{u-lambda-R}), and $C$ is a positive constant independent of both $R$ and $ \lambda$. 
Indeed, recalling that $c_{\lambda,R}\sim \lambda^{-\frac{2}{p-1}}$, we have that
\begin{align*}
\|u_{\lambda,R}\|_{L^\infty(\mathbb{R}^N)} &\le \lambda^\frac{1}{2_s^*-p-1}C_1\|u_{\lambda,R}\|_{L^{2_s^*(\Omega)}}^\frac{2_s^*-2}{2_s^*-p-1}\\
&= C_1\left( \lambda \|u_{\lambda,R}\|_{L^{2_s^*(\Omega)}}^{2_s^*-2}  \right)^\frac{1}{2_s^*-p-1}\\
&\le C \left( \lambda \|u_{\lambda,R}\|_{X^s_0(\Omega)}^{2_s^*-2}  \right)^\frac{1}{2_s^*-p-1}\\
&\le C \left( \lambda (c_{\lambda,R}^{\frac{1}{2}})^{2_s^*-2}  \right)^\frac{1}{2_s^*-p-1}\\
&\le C \left( {\lambda}^{1-\frac{2_s^*-2}{p-1}} \right)^\frac{1}{2_s^*-p-1}\\
&\le C \lambda^{-\frac{1}{p-1}}\\
 &\to 0\qquad \text{ as } \lambda \to +\infty,
\end{align*}
This means that there exists $\lambda^*>0$ such that $0\le u_{\lambda,R}\leq R<1$ for any $\lambda\geq \lambda^*$. 
Hence, according to Lemma \ref{lemma-conseceuncia-montana}, $u_{\lambda}:=u_{\lambda,R}$ becomes a nontrivial nonnegative  solution of the original problem (\ref{P}) since  $f(u_{\lambda})=f_R(u_{\lambda,R})$ when $0\le u_{\lambda}\le R$. Therefore, the proof is completed. 
\end{proof}
%

\setcounter{equation}{0}
\section{The second solution}\label{section4}

The aim of this section is to prove that if $\lambda>0$ is sufficiently large, then there exists a  {\em second} solution to problem (\ref{P}). 
We will reach the goal by means of the sub- and super-solution method. 

For convenience, we start by considering an auxiliary problem. 
Let $g\colon \Omega \times \mathbb{R}\to  \mathbb{R}$ de a function. 
In what follows we will assume on $g$ the following,
\begin{itemize}
	\item[$(G_1)$] $g(x,t)$ is a Carath\'eodory function (i.e. $g(\cdot,t)$ is measurable for all $t\in \mathbb{R}$ and $g(x,\cdot)$ is continuous for a.e. $x\in\Omega$), and $g(\cdot,t)$ is bounded if $t$ belongs to bounded sets.
	\item[$(G_2)$] There exists $M>0$ such that the map $t\mapsto g(x,t)+Mt$ is nondecreasing for a.e. $x\in\Omega$.
\end{itemize}

Consider now the auxiliary problem 
\begin{equation}\label{aux-2ndsolution}
\left\{\begin{array}[c]{lll}
(-\Delta)^s u = g(x,u) & \text{in }\Omega,\\
u=0 & \text{on } \Omega^c. 
\end{array}\right.
\end{equation} 
For dealing with (\ref{aux-2ndsolution}), it is convenient to introduce a precise definition of sub- and super-solution  in our context, which we do with the help of the following lemma.
\begin{lemma}{\cite[Lemma 1.26]{Molica-Radulescu-Servadei}}
Let $\varphi\in C_0^2(\Omega)$ and $u\in X^s(\Omega)\cap L^\infty (\Omega^c)$. Then, the following equalities hold true:
\begin{align*}
\int_{\mathbb{R}^N\times\mathbb{R}^N}(u(x)-u(y))&(\varphi(x)-\varphi(y))K(x-y)\,dx\,dy\\ &=\int_Q (u(x)-u(y))(\varphi(x)-\varphi(y))K(x-y)\,dx\,dy\\ &=\int_{\mathbb{R}^N\times\mathbb{R}^N}u(x)(2\varphi(x)-\varphi(x+y)-\varphi(x-y))K(y)\,dx\,dy.
\end{align*}
\end{lemma}
According the previous lemma, and since
\begin{align*}
(-\Delta)^s u(x)&=c(N,s)\,\mbox{P.V.} \int_{\mathbb{R}^N}\frac{u(x)-u(y)}{|x-y|^{N+2s}}dy\\ &=-\frac{1}{2}c(N,s)\int_{\mathbb{R}^N}\frac{u(x+y)+u(x-y)-2u(x)}{|y|^{N+2s}}dy,
\end{align*}
we can see that for any $u\in X^s(\Omega)\cap L^\infty(\Omega^c)$ and $\varphi\in C_0^2(\Omega)$ it is verified 
\begin{align*}
\int_{\Omega}(-\Delta)^su(x)\varphi(x)\,dx&=c(N,s)\,\mbox{P.V.} \int_{\mathbb{R}^N\times\mathbb{R}^N}(u(x)-u(y))(\varphi(x)-\varphi(y))K(x-y)\,dx\,dy\\ &=c(N,s)\,\mbox{P.V.}\int_Q (u(x)-u(y))(\varphi(x)-\varphi(y))K(x-y)\,dx\,dy\\ &=\int_{\Omega}u(x)(-\Delta)^s\varphi(x)\,dx.
\end{align*}
Now, we introduce the notion of sub- and super-solution associated to (\ref{aux-2ndsolution}) that we use here. 
\begin{definition}
Let $g:\Omega\times \mathbb{R}\to \mathbb{R}$ a Carath\'eodory function. We say that functions $\underline{u}\in L^1(\Omega)$ and $\overline{u}\in L^1(\Omega)$ are respectively a sub-solution and a super-solution of 
  (\ref{aux-2ndsolution}) if is respectively satisfied   
$$
\int_\Omega\underline{u}(x)(-\Delta)^s\phi(x)\,dx\leq\int_{\Omega}\phi(x)g(x,\underline{u}(x))\,dx
$$
and 
$$
\int_\Omega\overline{u}(x)(-\Delta)^s\phi(x)\,dx\geq\int_{\Omega}\phi(x)g(x,\overline{u}(x))\,dx,
$$
for all $\phi\in C_0^2(\Omega)$, $\phi\geq 0$ in $\Omega$. 
\end{definition}
Other important tool to prove the existence of a {\em second} nontrivial solution to (\ref{P}) is a weak comparison principle for a related problem. Despite the proof is straightforward, we include it for the reader convenience. 
\begin{lemma}[Weak comparison principle]\label{lemweakcomp} Let $\psi:\Omega\times \mathbb{R}\to \mathbb{R}$ be a Carath\'eodory function, such that  $\psi(x,0)=0$ in $\Omega$, $\psi(x,\cdot)$ is nondecreasing, and $0\leq \psi(x,t)\leq C(1+|t|)$ for all $(x,t)\in\Omega\times \mathbb{R}$, for some $C>0$.  If $u_1,u_2\in X^s(\Omega)\cap L^\infty(\Omega^c)$ are such that 
$$
\begin{array}[c]{lll}
\displaystyle \int_{\Omega}u_2(x)(-\Delta)^s\varphi(x)\,dx+\int_{\Omega}\psi(x,u_2(x))\varphi(x)\,dx\\ \displaystyle \qquad \leq \int_{\Omega}u_1(x)(-\Delta)^s\varphi(x)\,dx+\int_{\Omega}\psi(x,u_1(x))\varphi(x)\,dx
\end{array}
$$
for all $\varphi\in X_0^s(\Omega)$ with $\varphi\geq 0$, and 
$$
u_2\leq u_1\mbox{ on }\Omega^c, 
$$
then 
$$
u_2 \leq u_1\quad \mbox{ a.e. in }\Omega.
$$
\end{lemma}
\begin{proof}
Let us choose $\varphi=(u_2-u_1)^+\in X_0^s(\Omega)$. 
It follows that
\begin{align*}
0&\geq \int_{\Omega}(u_2(x)-u_1(x))(-\Delta)^s(u_2-u_1)^+(x)\,dx\\ &\qquad +\int_\Omega[\psi(x,u_2(x))-\psi(x,u_1(x))](u_2-u_1)^+(x)\,dx\\ &=c(N,s)\,\mbox{P.V.} \int_{\mathbb{R}^N\times\mathbb{R}^N}((u_2(x)-u_1(x))-(u_2(y)-u_1(y)))\cdot\\ &\qquad\cdot  ((u_2(x)-u_1(x))^+-(u_2(y)-u_1(y))^+)K(x-y)dx\,dy\\&\qquad +\int_\Omega[\psi(x,u_2(x))-\psi(x,u_1(x))](u_2-u_1)^+(x)\,dx\\ &=c(N,s)\,\mbox{P.V.} \int_{\mathbb{R}^N\times\mathbb{R}^N}\left((u_2(x)-u_1(x))^+-(u_2(y)-u_1(y))^+\right)^2 K(x-y)\,dx\,dy\\ 
&\qquad +c(N,s)\,\mbox{P.V.}\int_{\mathbb{R}^N\times\mathbb{R}^N}(u_2(x)-u_1(x))^+(u_2(y)-u_1(y))^- K(x-y)dx\,dy\\
&\qquad +c(N,s)\,\mbox{P.V.}\int_{\mathbb{R}^N\times\mathbb{R}^N}(u_2(x)-u_1(x))^-(u_2(y)-u_1(y))^+ K(x-y)\,dx\,dy\\
&\qquad +\int_\Omega[\psi(x,u_2(x))-\psi(x,u_1(x))](u_2-u_1)^+(x)\,dx.
\end{align*}
Since every term in the last expression is nonnegative, we obtain that $(u_2-u_1)^+=0$ a.e. in $\mathbb{R}^N$, which leads to  $u_2\leq u_1$ a.e. in $\mathbb{R}^N$. Therefore the proof is completed.  
\end{proof}
Consider now the auxiliary problem \eqref{aux-2ndsolution} with the particular choice $$g(x,u)=h(x)-\psi(x,u),$$ where $\psi$ verifies the hypotheses of Lemma \ref{lemweakcomp}, and $h$ is a function that belongs to $L^2(\Omega)$. 
Under these assumptions on $\psi$ and $h$, we will are able to prove existence of a solution for the associated equation for this $g$, which is the content of next lemma. 
\begin{lemma}\label{lemaT}
Let $\psi$ be a Carath\'eodory function satisfying the hypotheses of Lemma \ref{lemweakcomp}. 
Then, for every $h\in L^2(\Omega)$, the problem
\begin{equation}\label{Pg}
\left\{
\begin{array}[c]{lll}
(-\Delta)^s u+\psi(x,u) = h(x) &\mbox{in } \Omega, \\
u=0&\mbox{on } \Omega^c,
\end{array}
\right.
\end{equation}
admits a unique weak solution $u\in X_0^s(\Omega)$. 
Moreover, the associated operator $T:L^2(\Omega)\to X_0^s(\Omega)$, $h\mapsto u$ is nondecreasing, and strong-strong continuous, that is, if $h_k \to h$ strongly in $L^2(\Omega)$, then $T h_k \to Th$ strongly in $X^s_0(\Omega)$.
\end{lemma}
\begin{proof}
Let $u\in X_0^s(\Omega)$ and define 
$$
I(u)=\frac{1}{4}\int_{\mathbb{R}^N\times\mathbb{R}^N} |u(x)-u(y)|^2 K(x-y)\,dx\,dy+\int_{\Omega}\Psi(x,u)\,dx-\int_{\Omega}h(x)u(x)\,dx,
$$
where $\Psi(x,t)=\int_0^t\psi(x,\tau)\,d\tau \ge 0$. 
It is not difficult to show that the functional $I$ is coercive and weak lower semicontinuous. 
Indeed, by H\"older's inequality, Lemma \ref{lemma-inclusion}, and Young's inequality, we get
\begin{align*}
I(u)&\ge\frac{1}{4}\|u\|^2_{X_0^s(\Omega)}-\|h\|_{L^2(\Omega)}\|u\|_{L^2(\Omega)}\\
&\ge \left( \frac{1}{4} -\varepsilon C\right)\|u\|^2_{X_0^s(\Omega)}-C_{\varepsilon}\|h\|_{L^2(\Omega)}^2
\end{align*}
for all sufficiently small $\varepsilon>0$. 
Note that we can fix a nonnegative constants $B$ and a positive constant $C$ such that for every $\varepsilon>0$ sufficiently small  
$$
I(u) \ge C\|u\|_{X_0^s(\Omega)}^2 -B\quad\mbox{ for all }u\in X_0^s(\Omega).
$$
Therefore, $I$ is coercive. 

The lower semicontinuity follows from the lower semicontinuity of the ${X_0^s(\Omega)}$-norm, (iii) of Lemma \ref{lemma-inclusion}, and the fact that 
$$
\liminf_{k\to \infty}\int_{\Omega}\Psi(x,u_k)\,dx \ge \int_{\Omega}\Psi(x,u)\,dx
$$
if $u_k\rightharpoonup u$ in $X_0^s(\Omega)$, thanks to Fatou's Lemma. 
Notice that again by (iii) Lemma \ref{lemma-inclusion}, we get the strongly $L^2(\Omega)$ convergence by taking a subsequence. 

Consequently, by the direct method of calculus of variations, there exists a unique $u\in X_0^s(\Omega)$ such that 
$$
I(u)=\inf_{v\in X_0^s(\Omega)}I(v).
$$
The fact that $T$ is nondecreasing follows from Lemma \ref{lemweakcomp}. Let us see $T$ is strong-strong continuous. 

Let $\{h_k\}_{k\in  \mathbb{N} } \subset L^2(\Omega)$ be such that $h_k \to h$ strongly in $L^2(\Omega)$. 
Denote $u_k=Th_k$, $u=Th$. Since $\psi\ge 0$ and $u_k$ satisfies the equation in (\ref{Pg}), from H\"older's inequality and Lemma \ref{lemma-inclusion} we deduce that
$$
\|u_k\|_{X^s_0(\Omega)}^2 \le \|h_k\|_{L^2(\Omega)} \|h_k\|_{L^2(\Omega)} \le C\|h_k\|_{L^2(\Omega)} \|u_k\|_{X^s_0(\Omega)} \le C\|u_k\|_{X^s_0(\Omega)}.
$$
Then, $\{u_k\}_{k\in \mathbb{N}} \subset X_0^s(\Omega)$ is a bounded sequence. 
Thus, there exists a subsequence $\{u_{k_j}\}_{j\in \mathbb{N}}$ such that $u_{k_j} \rightharpoonup v$ weakly in $X_0^s(\Omega)$. 
Thanks to Lemma \ref{lemma-inclusion}, we can assume in addition $u_{k_j}\to v$ strongly in $L^2(\Omega)$ and a.e. in $\Omega$ by taking another subsequence if necessary.

Let $\varphi \in X_0^s(\Omega)$ and consider it as a test function for the equation in (\ref{Pg}) associated to $h_k$. Then,
$$
\begin{array}[c]{lll}
\displaystyle \frac{1}{2}\int_{\mathbb{R}^N\times \mathbb{R}^N}\!(u_{k_j}(x)-u_{k_j}(y))(\varphi(x)-\varphi(y))K(x-y)\, dx\,dy + \int_{\Omega} \psi(x,u_{k_j}(x))\varphi(x) \, dx\smallskip\\ 
\displaystyle \qquad = \int_{\Omega} h_{k_j}\varphi \, dx.
\end{array}
$$ 
Since  $h_k \to h$ strongly in $L^2(\Omega)$, $u_{k_j} \rightharpoonup v$ weakly in $X_0^s(\Omega)$, $u_{k_j}\to v$ strongly in $L^2(\Omega)$ and a.e. in $\Omega$ as $k\rightarrow +\infty$; and bearing in mind that $\psi$ is a Carath\'eodory function, by taking the limit $j\to \infty$ we get
$$
\frac{1}{2}\int_{\mathbb{R}^N\times \mathbb{R}^N}\!(v(x)-v(y))(\varphi(x)-\varphi(y))K(x-y)\, dxdy + \int_{\Omega} \psi(x,v(x))\varphi(x) \, dx = \int_{\Omega} h\varphi \, dx.
$$ 
Due to the uniqueness of solutions to this problem, we get $v=u$. 
Notice that the limit function does not depend on the election of the subsequence. 
Therefore, $u_k \rightharpoonup u$ weakly in $X_0^s(\Omega)$.

To deduce the strong convergence in $X_0^s(\Omega)$, notice that 
$$
\| u_k -u \|_{X_0^s(\Omega)}^2 = \int_{\Omega}(h_k-h)(u_k-u)\, dx - \int_{\Omega}(\psi(x,u_k)-\psi(x,u))(u_k-u)\, dx.
$$
Hence, by taking the limit as $k\rightarrow +\infty$, we get that $h_k \to h$ strongly in $L^2(\Omega)$, $u_{k}\to u$ strongly in $L^2(\Omega)$ and a.e. in $\Omega$; and since $\psi$ is a Carath\'eodory function, we get $u_k \to u$ strongly in $X_0^s(\Omega)$, which implies the strong-strong continuity of $T$. 
\end{proof}

\begin{theorem}\label{subandsuper}
Let $g\colon \Omega \times \mathbb{R}\to  \mathbb{R}$ be such that hypotheses $(G_1)$-$(G_2)$ are satisfied. 
Consider the problem 
\begin{equation}\label{PL}
\left\{
\begin{array}[c]{lll}
(-\Delta)^s u = g(x,u) &\mbox{in }\Omega, \smallskip\\
u=0&\mbox{on }  \mathbb{R}^N\setminus\Omega.
\end{array}
\right.
\end{equation}
Assume there exist $\underline{u},\overline{u}\in X^s(\Omega)\cap L^\infty(\Omega^c)$, a sub-solution and a super solution, respectively, with  $\underline{u}(x)\leq\overline{u}(x)$ a.e. in $\Omega$. 
Then, there exists a minimal (and, respectively, a maximal) weak solution $u_*$ (resp. $u^*$) for the problem (\ref{PL}) in the ``interval'' 
$$
[\underline{u},\overline{u}]=\{u\in L^\infty(\Omega): \underline{u}(x)\leq u(x)\leq \overline{u}(x) \mbox{ a.e. in }\Omega\}.
$$ 
\end{theorem}
\begin{proof}
	Consider the set $[\underline{u},\overline{u}]$ with the topology
	of convergence a.e., and define the operator
	$S:[\underline{u},\overline{u}]\to L^{2}(\Omega)$
	by
	$$
	Sv(\cdot)=g(\cdot,v(\cdot))+Mv(\cdot)\in
	L^\infty(\Omega) \subset L^2(\Omega)
	$$
	for any $v\in[\underline{u},\overline{u}]$. By $(H_1)$-$(H_2)$, we
	find that $S$ is nondecreasing and bounded. 
	
	Moreover, if
	$v_n,v\in[\underline{u},\overline{u}]$ are such that 

$v_n\to v$ a.e. in $\Omega$. Thanks to $(H_1)$, $g(x, \cdot)$ is continuous. Therefore, $g(x, v_n(x)) \to g(x, v(x))$ in $ \mathbb{R}$, a.e. $x\in\Omega$. Clearly, $Mv_n(x)\to Mv(x)$ a.e. $x\in\Omega$. Again, by $(H_1)$ and the Dominated Lebesgue Convergence Theorem,  $S(v_n)\to S(v)$ strongly in $L^2(\Omega)$. 

We have proved the a.e-strong continuity of the operator $S$, that is, $S(v_n)\to S(v)$ strongly in $L^2(\Omega)$, if $v_n\to v$ a.e. in $\Omega$, with $v_n,v \in [\underline{u},\overline{u}]$.
	
Consider $H:[\underline{u},\overline{u}]\to X^{s}_0(\Omega)$ the nondecreasing operator defined by $H=T\circ S$, where $T$ is given in Lemma \ref{lemaT}.  That is, for a function
$v\in[\underline{u},\overline{u}]$, $H(v)$ is the unique weak
solution of the boundary value problem
$$
\begin{cases}
(-\Delta)^s u+Mu=g(x,v)+Mv &\mbox{in }\Omega,
\\
\hspace{2.2cm} u=0 & \mbox{in }\Omega^c.
\end{cases}
$$
Notice that $H$ is a.e-strong continuous, that is, $H(v_n)\to H(v)$ strongly in $X_0^s(\Omega)$, if $v_n\to v$ a.e. in $\Omega$, with $v_n,v \in [\underline{u},\overline{u}]$.

Let $u_1=H(\underline{u})$ and $u^1=H(\overline{u})$. Then, for every
$\phi\in X^s_0(\Omega)$ with $\phi\geq 0$ we have that
\begin{align*}
\int_\Omega u_1(-\Delta)^s\phi &
+\int_\Omega Mu_1\phi
\\
& =
\int_\Omega(g(x,\underline{u})+M\underline{u})\phi
\\
& \geq \int_\Omega\underline{u}(-\Delta)^s\phi+\int_\Omega M\underline{u}\phi
\\
\intertext{and} \int_\Omega 	u^1(-\Delta)^s\phi & +\int_\Omega Mu^1\phi
\\
& =
\int_\Omega(g(x,\overline{u})+M\overline{u})\phi\\
& \leq \int_\Omega \overline{u}(-\Delta)^s\phi+\int_\Omega M\overline{u}\phi.
\end{align*}
Applying Lemma \ref{lemweakcomp}, and taking into account that $H$ is
nondecreasing, we obtain that $\underline{u}\leq
H(\underline{u})\leq H(u)\leq H(\overline{u})\leq\overline{u}$,
a.e. in $\Omega$, for any $u\in[\underline{u},\overline{u}]$. 
	
By the same reasoning, we can prove the existence of sequences
$\{u^n\}_{n\in  \mathbb{N}}$ and $\{u_n\}_{n\in  \mathbb{N}}$ satisfying
\begin{align*}
& u^0 =\overline{u},\quad u^{n+1}=H(u^n),
\\
& u_0=\underline{u},\quad u_{n+1}=H(u_n),
\end{align*}
and, for every weak solution $u\in[\underline{u},\overline{u}]$ of
(PL), we have that
$$
u_0\leq u_1\leq\ldots\leq u_n\leq u\leq u^n\leq\ldots\leq u^1\leq
u^0\mbox{ a.e. in }\Omega.
$$
Then, $u_n\to u_*$, $u^n\to u^*$, a.e. in $\Omega$, with
$u_*,u^*\in[\underline{u},\overline{u}]$, $u_*\leq u^*$ a.e. in
$\Omega$. Since $u_{n+1}=H(u_n)\to H(u_*)$, and $u^{n+1}=H(u^n)\to
H(u^*)$ in $X_0^s(\Omega)$ by the continuity of $H$,
we find that $u_*,u^*\in X_0^s(\Omega)$ with
$u_*=H(u_*)$, $u^*=H(u^*)$. This completes the proof. 
\end{proof}

Now, we are able to prove our main result, Theorem \ref{main-theorem0}. 

\begin{proof}[Proof of Theorem \ref{main-theorem0}.]
Accoding Theorem \ref{main}, there exists a $\lambda^*$ such that the problem (\ref{P}) admits a solution $u_\lambda$ for any $\lambda>\lambda^*$. Moreover, following the arguments given in the proof of Theorem \ref{main},  
$$
\|u_\lambda\|_\infty<1 \quad  \text{for all } \lambda>\lambda^*.
$$ 
Hence, if we fix $\lambda_0>\lambda^*$, then $\underline{u}=u_{\lambda_0}$ is a sub-solution of problem (\ref{P}) for any $\lambda>\lambda_0$. Indeed, given $\phi \in C_0^2(\Omega)$ such that $\phi \ge 0$ in $\Omega$, we have
\begin{align*}
\int_{\Omega} u_{\lambda_0} (-\Delta)^s \phi \, dx &= \int_{\Omega} \phi  (-\Delta)^s u_{\lambda_0}\, dx \\
&= \lambda_0 \int_{\Omega} f(u_{\lambda_0}) \phi \, dx \\
&< \lambda \int_{\Omega} f(u_{\lambda_0}) \phi \, dx.
\end{align*}
On the other hand, $\overline{u}=1$ is a super-solution of the problem (\ref{P}) for any $\lambda>\lambda_0$. Indeed,  
\begin{align*}
\int_{\Omega}  (-\Delta)^s \phi \, dx &\ge \int_{\Omega} u_{\lambda} (-\Delta)^s  \phi  \, dx \\
&= \lambda \int_{\Omega} f(u_{\lambda}) \phi \, dx\\
&\geq  0\\
&=\lambda \int_{\Omega} f(1) \phi \, dx.
\end{align*}
Hence, by  Theorem \ref{subandsuper}, if we put $g(x,u)= \lambda f(u)$ for any $\lambda> \lambda_0$,  then we obtain another solution $v_\lambda$ of (\ref{P}) for every $\lambda>\lambda_0$. Observe that from Theorem \ref{main}, 
$$
\lim_{\lambda \to +\infty}\|u_\lambda\|_{L^{\infty}(\Omega)}=0.
$$
Therefore, $\lambda_0$ can be chosen sufficiently large, so that $u_\lambda<\underline{u}\leq v_\lambda\leq 1=\bar{u}$. In this way, the proof is completed  if we take $\overline{\lambda}=\lambda_0$. 
\end{proof}


\end{document}